\documentclass{amsart}
\usepackage{amssymb,epsfig,mathrsfs,mathpazo}
\usepackage{color}
\usepackage{epsfig,bm}
\usepackage[multiple]{footmisc}

\newtheorem{theorem}{Theorem}[section]
\newtheorem{lemma}[theorem]{Lemma}

\theoremstyle{definition}

\theoremstyle{remark}
\newtheorem{remark}[theorem]{Remark}
\numberwithin{equation}{section}

\newcommand{\eps}{\varepsilon}

\newcommand{\R}{\mathbb R}

\newcommand{\N}{\mathbb N}

\newcommand{\cL}{\mathcal L}
\newcommand{\Lis}{\cL\mathrm{is}}

\newcommand{\identity}{\mathrm{Id}}

\DeclareMathOperator{\ran}{ran}

\DeclareMathOperator{\clos}{clos}

\newcommand{\be}{\begin{equation}}
\newcommand{\ee}{\end{equation}}

\newcommand{\tria}{{\mathcal T}}

\newcommand{\aumlaut}{{\"a}}
\newcommand{\uumlaut}{{\"u}}

\title[Stability of Galerkin discretizations of parabolic evolution equations]{Stability of Galerkin discretizations of a mixed space-time variational formulation of parabolic evolution equations}

\date{\today}

\author{Rob Stevenson and Jan Westerdiep}

\address{
Korteweg-de Vries (KdV) Institute for Mathematics, University of Amsterdam, P.O. Box 94248, 1090 GE Amsterdam, The Netherlands.
}
\email{r.p.stevenson@uva.nl, j.h.westerdiep@uva.nl}

\thanks{The first author has been supported by NSF Grant DMS 172029.
The second author has been supported by the Netherlands Organization for Scientific Research
(NWO) under contract.~no.~613.001.652}



\begin{document}

\begin{abstract}  
We analyze Galerkin discretizations of a new well-posed mixed space-time variational formulation of parabolic PDEs. For suitable pairs of finite element trial spaces, the resulting Galerkin operators are shown to be uniformly stable.
The method is compared to two related space-time discretization methods introduced in 
[\emph{IMA J. Numer. Anal.}, 33(1) (2013), pp. 242-260] by R. Andreev
and in [\emph{Comput. Methods Appl. Math.}, 15(4) (2015), pp. 551-566] by O.~Steinbach.
\end{abstract}

\maketitle

\section{Introduction}
In recent years one witnesses a rapidly growing interest in simultaneous space-time methods for solving parabolic evolution equations originally introduced in \cite{18.63,18.64}, see e.g.~\cite{77.5,11,299,249.2,75.27,169.05,247.155,64.577,234.7,243.867,310.6,249.3,75.257}.
Compared to classical time marching methods, space-time methods 
are much better suited for a massively parallel implementation, and have the potential to drive adaptivity simultaneously in space and time.

Apart from the first order system least squares formulation recently introduced in \cite{75.257}, the known well-posed simultaneous space-time variational formulations of parabolic equations in terms of partial differential operators only, so not involving non-local operators, are not coercive.
As a consequence, it is non-trivial to find families of pairs of discrete trial- and test-spaces for which the resulting Petrov-Galerkin discretizations are uniformly stable. The latter is a sufficient and, as we will see, necessary condition for the Petrov-Galerkin approximations to be {\em quasi-optimal}, i.e., to yield an up to a constant factor best approximation to the solution from the trial space.
This concept has to be contrasted to rate optimality that, for quasi-uniform temporal and spatial partitions, has been shown for any reasonable numerical scheme under the assumption of sufficient regularity of the solution.

If one allows different spatial meshes at different times, then for the classical time marching schemes quasi-optimality of the numerical approximations is known not to be guaranteed as demonstrated in \cite[Sect.~4]{70.1}.

In view of the difficulty in constructing stable pairs of trial- and test-spaces,  
 in \cite{11} Andreev considered minimal residual Petrov-Galerkin discretizations.
They have an equivalent interpretation as Galerkin discretizations of an extended self-adjoint mixed system, with the Riesz lift of the residual of the primal variable being the secondary variable.
This is the point of view we will take.

A different path was followed by Steinbach in \cite{249.2}. 
Assuming a homogenous initial condition, for equal test and trial finite element spaces w.r.t.~fully
general finite element meshes, there stability was shown w.r.t.~a weaker mesh-dependent norm on the trial space. 
As we will see, however, this has the consequence that for some solutions of the parabolic problem these Galerkin approximations are far from being quasi-optimal w.r.t.~the natural mesh-independent norm on the trial space.

In the current work, we modify Andreev's approach by considering an equivalent but simpler mixed system that we construct from a space-time variational formulation that follows from applying the Br\'{e}zis-Ekeland-Nayroles principle \cite{35.81,233.5}.
With the same trial space for the primal variable, we show stability of the Galerkin discretization of this mixed system whilst utilizing a smaller trial space for the secondary variable.
In addition, the stiffness matrix resulting from this mixed system is more sparse.
In our numerical experiments the errors in the Galerkin solutions are nevertheless very comparable.

\subsection{Organization} In Sect.~\ref{S2} we derive the two self-adjoint mixed system formulations of the parabolic problem that are central in this work.
In Sect.~\ref{S3} we give sufficient conditions for stability of Galerkin discretizations for both systems.
We provide an a priori error bound for the Galerkin discretization of the newly introduced system, and improved 
a priori error bounds for the methods from \cite{11} and \cite{249.2}.
In Sect~\ref{S4}, we show that the crucial condition for stability (being the only condition for the newly introduced mixed system) is satisfied for prismatic space-time finite elements
whenever the generally non-uniform partition in time is independent of the spatial location, and the generally non-uniform spatial mesh in each time slab is such that the corresponding $L_2$-orthogonal projection is uniformly $H^1$-stable.
 In Sect.~\ref{Snum} we present some first simple numerical experiments for a one-dimensional spatial domain and uniform meshes.
Conclusions are presented in Sect.~\ref{Sconcl}.

\subsection{Notations}
In this work, by $C \lesssim D$ we will mean that $C$ can be bounded by a multiple of $D$, independently of parameters which C and D may depend on. Obviously, $C \gtrsim D$ is defined as $D \lesssim C$, and $C\eqsim D$ as $C\lesssim D$ and $C \gtrsim D$.

For normed linear spaces $E$ and $F$, by $\cL(E,F)$ we will denote the normed linear space of bounded linear mappings $E \rightarrow F$, 
and by $\Lis(E,F)$ its subset of boundedly invertible linear mappings $E \rightarrow F$.
We write $E \hookrightarrow F$ to denote that $E$ is continuously embedded into $F$.
For simplicity only, we exclusively consider linear spaces over the scalar field $\R$.

For linear spaces $E$ and $F$, sequences $\Phi=(\phi_j)_{j \in J} \subset E$, $\Psi=(\psi_i)_{i \in I} \subset F$, $f \in F^*$, and  a linear $A\colon E \rightarrow F^*$, 
we define the column vector $f(\Psi):=[f(\psi_i)]_{i \in I}$ and matrix $(A \Phi)(\Psi):=[(A \phi_j)(\psi_i)]_{i \in I,j \in J}$.
If $E=F$ is an inner product space, then with $R\colon E \rightarrow E'$ denoting the Riesz map, we set $\langle\Psi,\Phi\rangle:=(R \Phi)(\Psi)=[(R \phi_j)(\psi_i)]_{i \in I,j \in J}=
[\langle \psi_i,\phi_j\rangle]_{i \in I,j \in J}$.

\section{Space-time formulations of the parabolic evolution problem} \label{S2}
Let $V,H$ be separable Hilbert spaces of functions on some ``spatial domain'' 
such that $V \hookrightarrow H$ with dense and compact embedding. 
Identifying $H$ with its dual, we obtain the Gelfand triple 
$V \hookrightarrow H \simeq H' \hookrightarrow V'$. 

We use the notation 
$\langle \cdot,\cdot \rangle$ to denote both the scalar product on 
$H \times H$, and its unique extension by continuity 
to the duality pairing on $V' \times V$.
Correspondingly, the norm on $H$ will  be denoted by $\|\,\|$.

For a.e. 
$$
t \in I:=(0,T),
$$
let $a(t;\cdot,\cdot)$ denote a bilinear form on $V \times V$ such that for 
any $\eta,\zeta \in V$, $t \mapsto a(t;\eta,\zeta)$ is measurable on $I$, 
and such that for 
a.e. $t\in I$,
\begin{alignat}{3} \label{1}
|a(t;\eta,\zeta)| 
& \lesssim  \|\eta\|_{V} \|\zeta\|_{V} \quad &&(\eta,\zeta \in V) \quad &&\text{({\em boundedness})},
\\ \label{2}
 a(t;\eta,\eta)  &\gtrsim \|\eta\|_{V}^2 \quad 
&&(\eta \in {V}) \quad &&\text{({\em coercivity})}.
\end{alignat}

With $A(t) \in \Lis({V},V')$ being defined by $ (A(t) \eta)(\zeta)=a(t;\eta,\zeta)$, we are interested in solving the {\em parabolic initial value problem} to finding $u$ such that
\begin{equation} \label{11}
\left\{
\begin{array}{rl} 
\frac{d u}{d t}(t) +A(t) u(t)&\!\!\!= g(t) \quad(t \in I),\\
u(0) &\!\!\!= u_0.
\end{array}
\right.
\end{equation}

\begin{remark} With $\tilde{u}(t):=u(t) e^{-\varrho t}$, \eqref{11} is equivalent to $
\frac{d \tilde{u}}{d t}(t) +(A(t)+\varrho \identity) \tilde{u}(t)= g(t)e^{-\varrho t}$  ($t \in I$),
$\tilde{u}(0) = u_0$. So if initially $a(t;\eta,\eta)$ is not coercive but only satisfies a {\em G{\aa}rding inequality}
$ a(t;\eta,\eta) + \varrho \langle \eta,\eta\rangle  \gtrsim \|\eta\|_{V}^2$ ($\eta \in {V}$), then one can consider a transformed problem such that \eqref{2} is valid.
\end{remark}

In a simultaneous space-time variational formulation, the parabolic PDE reads as finding $u$ from a suitable space of functions of time and space such that
\begin{equation} \label{8}
(Bw)(v):=\int_I 
\langle { \textstyle \frac{d w}{dt}}(t), v(t)\rangle + 
a(t;w(t),v(t)) dt = \int_I 
\langle g(t), v(t)\rangle =:g(v)
\end{equation}
for all $v$ from another suitable space of functions of time and space.
One  possibility to enforce the initial condition is by testing it against additional test functions. A proof of the following result can be found in \cite{247.15}, cf. 
\cite[Ch.XVIII, \S3]{63} and \cite[Ch.~IV, \S26]{314.9} for slightly different statements.

\begin{theorem} \label{thm0} With $X:=L_2(I;{V}) \cap H^1(I;V')$, $Y:=L_2(I;{V})$,
under conditions \eqref{1} and \eqref{2} it holds that
\be \label{var1}
\left[\begin{array}{@{}c@{}} B \\ \gamma_0\end{array} \right]\in \Lis(X,Y' \times H),
\ee
where for $t \in \bar{I}$, $\gamma_t\colon u \mapsto u(t,\cdot)$ denotes the trace map.
That is, assuming $g \in Y'$ and $u_0 \in H$, finding $u \in X$ such that 
\be \label{exact}
(Bu)(v_1)+\langle u(0,\cdot),v_2\rangle=g(v_1)+\langle u_0,v_2\rangle\quad ((v_1,v_2) \in Y \times H),
\ee
is a well-posed variational formulation of \eqref{11}.
\end{theorem}

One ingredient of the proof of this theorem is the continuity of the embedding $X \hookrightarrow C(\bar{I},H)$, in particular implying that for any $t \in \bar{I}$, $\gamma_t \in \cL(X,H)$.

Defining $A, A_s \in \Lis(Y,Y')$ (here \eqref{2} is used), $A_a \in \cL(Y,Y')$, and $C, \partial_t \in \cL(X,Y')$ by 
\begin{align*}
&(Au)(v):=\int_I a(t;u(t),v(t))\,dt,\quad
A_s:={\textstyle  \frac12}(A+A'), \quad A_a:={\textstyle \frac12}(A-A'),\\
&C:=B-A_s,\quad \partial_t:=B-A,
\end{align*}
an equivalent well-posed variational formulation of the parabolic PDE is obtained by applying the so-called Br\'{e}zis-Ekeland-Nayroles variational principle \cite{35.81,233.5}, cf. also \cite[\S3.2.4]{10.6}.
It reads as 
\be \label{var2}
(C' A_s^{-1} C+A_s+\gamma_T' \gamma_T)u=(\identity +C' A_s^{-1})g +\gamma_0' u_0,
\ee
where the operator at the left hand side is in $\Lis(X,X')$, is self-adjoint and coercive. 

We provide a direct proof of these facts.
Since $\left[\begin{array}{@{}cc@{}} A_s & 0 \\ 0 & \identity \end{array}\right] \in \Lis(Y\times H,Y'\times H)$, an equivalent formulation of \eqref{var1} as a self-adjoint saddle point equation reads as finding $(\mu,\sigma,u) \in Y\times H\times X$ (where $\mu$ and $\sigma$ will be zero) such that
\begin{align} \label{var4}
\left[\begin{array}{@{}ccc@{}} A_s & 0 & B\\ 0 & \identity & \gamma_0\\ B' & \gamma_0' & 0\end{array}\right]
\left[\begin{array}{@{}c@{}} \mu \\ \sigma \\ u \end{array}\right]&=
\left[\begin{array}{@{}c@{}} g \\ u_0 \\ 0 \end{array}\right],
\intertext{or} \label{var5}
(B' A_s^{-1} B+\gamma_0'\gamma_0)u&=B' A_s^{-1}g+\gamma_0' u_0.
\end{align}
Thanks to \eqref{var1}, this Schur complement $B' A_s^{-1} B+\gamma_0'\gamma_0$ is in $\Lis(X,X')$, is self-adjoint and coercive.

We show that \eqref{var5} and \eqref{var2} are equal. Recalling the definitions of $C$ and $\partial_t$, note that the right-hand sides of both equations are the same, and that
$$
B' A_s^{-1} B+\gamma_0'\gamma_0=C' A_s^{-1} C+A_s+C+C'+\gamma_0'\gamma_0=C' A_s^{-1} C+A_s+\partial_t+\partial_t'+\gamma_0'\gamma_0
$$
thanks to $A_a'=-A_a$.
The proof of our claim is completed by noting that for $w,v \in X$,
\begin{align*}
((\partial_t+\partial_t'+\gamma_0'\gamma_0)w)(v)&=
\int_I \langle { \textstyle \frac{d w}{dt}}(t), v(t)\rangle+\langle w(t), { \textstyle \frac{d v}{dt}}(t)\rangle\,dt+\langle w(0), v(0)\rangle\\
&=
\int_I {\textstyle \frac{d}{dt}} \langle w(t),v(t)\rangle\,dt+\langle w(0),v(0)\rangle=(\gamma_T'\gamma_Tw)(v).
\end{align*}

As \eqref{var5} was obtained as the Schur complement equation of \eqref{var4}, in its form \eqref{var2} it is naturally obtained as the Schur complement of 
 the problem of finding $(\lambda,u)\in Y \times X$ such that
 \be \label{var3}
 \left[\begin{array}{@{}cc@{}} A_s & C \\ C'  & -(A_s+\gamma_T' \gamma_T) \end{array}\right]
 \left[\begin{array}{@{}c@{}} \lambda \\ u \end{array}\right]=
 \left[\begin{array}{@{}c@{}} g \\ -(g+\gamma_0' u_0) \end{array}\right].
\ee
Knowing that its Schur complement is in $\Lis(X,X')$, $A_s \in \Lis(Y,Y')$, and $C \in \cL(X,Y')$, 
we infer that the self-adjoint operator at the left hand side of \eqref{var3} is in $\Lis(Y\times X,Y'\times X')$.

Substituting $C=B-A_s$ and $Bu=g$, we find that the secondary variable satisfies
$$
\lambda=u.
$$

\begin{remark} When reading $\gamma_T' \gamma_T$ as $\partial_t+\partial_t'+\gamma_0'\gamma_0$, the 
 system \eqref{var3} has remarkable similarities to a certain preconditioned version presented in \cite{234.7} of a discretized parabolic PDE using the implicit Euler method in time.
 Ideas concerning optimal preconditioning developed in that paper, as well as those in \cite{12.5},
can be expected to be applicable to Galerkin discretizations of \eqref{var3}.
\end{remark}

\begin{remark} In equations \eqref{var4} and \eqref{var5}, the operator $A_s$ can be replaced by a general self-adjoint $\tilde{A}_s \in \Lis(Y,Y')$. With $\tilde{C}:=B-\tilde{A}_s$, the equivalent equation \eqref{var2} then reads as
$$
(\tilde{C}' \tilde{A}_s^{-1} \tilde{C}+2 A_s-\tilde{A}_s+\gamma_T' \gamma_T)u=(\identity +\tilde{C}' \tilde{A}_s^{-1})g +\gamma_0' u_0,
$$
and \eqref{var3} as
$$
 \left[\begin{array}{@{}cc@{}} \tilde{A}_s & \tilde{C} \\ \tilde{C}'  & -(2 A_s-\tilde{A}_s+\gamma_T' \gamma_T) \end{array}\right]
 \left[\begin{array}{@{}c@{}} \lambda \\ u \end{array}\right]=
 \left[\begin{array}{@{}c@{}} g \\ -(g+\gamma_0' u_0) \end{array}\right],
$$
with solution $\lambda=u$.
\end{remark}

In the next section, we study Galerkin discretizations of equations \eqref{var4} and \eqref{var3}, which then are no longer equivalent.

Since the secondary variables $\mu$ and $\sigma$ in \eqref{var4} are zero, the subspaces for their approximation do not have to satisfy any approximation properties.
Since the secondary variable $\lambda$ in \eqref{var3} is non-zero, the subspace of $Y$ for its approximation has to satisfy approximation properties, and the error in its best approximation enters the upper bound for the error in the primal variable $u$.

On the other hand, (uniform) stability will be easier to realize with equation \eqref{var3} and will also be proven to hold true for $A_a \neq 0$; the system matrix will be more sparse; and the number of unknowns will be smaller.

In order to facilitate the derivation of some quantitative results, we will equip the spaces $Y$ and $X$ with the `energy-norms' defined by
$$
\|v\|^2_Y:=(A_s v)(v),\quad
\|u\|_{X}^2:=\|u\|_{Y}^2 +\|\partial_t u\|_{Y'}^2  + \|u(T)\|^2,
$$
which are equivalent to the standard norms on these spaces.
Correspondingly, orthogonality in $Y$ will be interpreted w.r.t.~the `energy scalar product' $(A_s\cdot)(\cdot)$.

\section{Stable discretizations of the parabolic problem} \label{S3}
\subsection{Uniformly stable (Petrov-) Galerkin discretizations and quasi-optimal approximations} \label{Sgeneral}
This subsection is devoted to proving the following theorem.

\begin{theorem}
Let $W$ and $Z$ be Hilbert spaces, and $F \in \Lis(Z,W')$.
Let $(W^\delta,Z^\delta)_{\delta \in \Delta}$ be a family of closed subspaces of $W \times Z$ such that for each $\delta \in \Delta$ it holds that 
${E_W^\delta}' F E_Z^\delta \in \Lis(Z^\delta,{W^\delta}')$, where $E_W^\delta\colon W^\delta \rightarrow W$, $E_Z^\delta\colon Z^\delta \rightarrow Z$ denote the trivial embeddings.
Then the collection $(z^\delta)_{\delta \in \Delta}$ of Petrov-Galerkin approximations to $z \in Z$, determined by
${E_W^\delta}' F E_Z^\delta z^\delta={E_W^\delta}' F z$, is \emph{quasi-optimal}, i.e. $\|z-z^\delta\|_Z \lesssim \inf_{0 \neq \bar{z}^\delta \in Z} \|z-\bar{z}^\delta \|_Z$, 
uniformly in $z \in Z$ and $\delta \in \Delta$, if and only if 
$$
\inf_{\delta \in \Delta}\inf_{0 \neq z \in Z^\delta} \sup_{0 \neq w \in W^\delta} \frac{|(Fz)(w)|}{\|z\|_Z \|w\|_W}>0\qquad\text{\emph{(uniform stability)}}.
$$
\end{theorem}

\begin{proof}
The mapping $P^\delta:=z \mapsto  z^\delta=E_Z^\delta ({E_W^\delta}' F E_Z^\delta)^{-1} {E_W^\delta}' F z$ is a projector.
For $\{0\} \subsetneq Z^\delta \subsetneq Z$, it holds that $P^\delta \not\in \{0,\identity\}$, and consequently, $\|\identity -P^\delta\|_{\cL(Z,Z)}=\|P^\delta\|_{\cL(Z,Z)}$ (see \cite{169.5,315.7}).
We obtain that 
\be \label{stablePG1}
\begin{split}
&\sup_{z \in Z\setminus Z^\delta}\frac{\|z-z^\delta\|_Z}{\inf_{\bar{z}^\delta \in Z^\delta}\|z-\bar{z}^\delta\|_Z}=
\sup_{z \in Z\setminus Z^\delta}\sup_{\bar{z}^\delta \in Z^\delta}\frac{\|(I-P^\delta)z\|_Z}{\|z-\bar{z}^\delta\|_Z}\\
&=
\sup_{0 \neq \bar{z} \in Z}\frac{\|(I-P^\delta)\bar{z}\|_Z}{\|\bar{z}\|_Z}=\|P^\delta\|_{\cL(Z,Z)}.
\end{split}
\ee
It remains to show uniform boundedness of $\|P^\delta\|_{\cL(Z,Z)}$ if and only if  uniform stability is valid.

The definition of $P^\delta$ shows that 
$$
 \|F^{-1}\|_{\cL(W',Z)}^{-1} \leq \frac{\|P^\delta\|_{\cL(Z,Z)}}{\|E_Z^\delta ({E_W^\delta}' F E_Z^\delta)^{-1} {E_W^\delta}'\|_{\cL(W',Z)}} \leq  \|F\|_{\cL(Z,W')}.
$$
Further, we have that
$$
\|E_Z^\delta ({E_W^\delta}' F E_Z^\delta)^{-1} {E_W^\delta}'\|_{\cL(W',Z)}=\| ({E_W^\delta}' F E_Z^\delta)^{-1} {E_W^\delta}'\|_{\cL(W',Z^\delta)}=\| ({E_W^\delta}' F E_Z^\delta)^{-1}\|_{\cL({W^\delta}',Z^\delta)}
$$
where the last equality follows from $\|{E_W^\delta}'\|_{\cL(W',{W^\delta}')} \leq 1$ and, for the other direction, from the fact that for given 
$f^\delta \in {W^\delta}'$ the function $f \in W'$ defined by $f|_{W^\delta}:=f^\delta$ and $f|_{(W^\delta)^\perp}:=0$ satisfies $\|f\|_{W'}=\|f^\delta\|_{{W^\delta}'}$ and $f^\delta={E_W^\delta}' f$.

 The proof is completed by
\be \label{stablePG3}
\|({E_W^\delta}' F E_Z^\delta)^{-1}\|_{\cL({W^\delta}',Z^\delta)}^{-1}=\inf_{0 \neq z \in Z^\delta} \sup_{0 \neq w \in W^\delta} \frac{|(Fz)(w)|}{\|z\|_Z \|w\|_W}. \qedhere
\ee
\end{proof}

\begin{remark} \label{remmie} In particular above analysis provides a short self-contained proof of the quantitative results
$$ \|F^{-1}\|_{\cL(W',Z)}^{-1} \leq \frac{\sup_{z \in Z\setminus Z^\delta}\frac{\|z-z^\delta\|_Z}{\inf_{\bar{z}^\delta \in Z^\delta}\|z-\bar{z}^\delta\|_Z}}
{\inf_{0 \neq z \in Z^\delta} \sup_{0 \neq w \in W^\delta} \frac{|(Fz)(w)|}{\|z\|_Z \|w\|_W}} \leq \|F\|_{\cL(Z,W')},
$$
that were established earlier in \cite[\S2.1, in particular  (2.12)]{258.4}.
\end{remark}

\subsection{Uniformly stable Galerkin discretizations of \eqref{var3}} \label{Sstablevar3}
Let $Y^\delta \times X^\delta$ be a closed subspace of $Y \times X$, and let
$E_Y^\delta\colon Y^\delta \rightarrow Y$ and  $E_X^\delta\colon X^\delta \rightarrow X$ denote the trivial embeddings.
Since ${E^\delta_Y}'A_s E^\delta_Y \in \Lis(Y^\delta,{Y^\delta}')$ (as well as being an isometry), 
the Galerkin operator resulting from \eqref{var3} can be factorized as
\be \label{factorized}
\begin{split}
&\left[\begin{array}{@{}cc@{}} {E^\delta_Y}'A_s E^\delta_Y & {E^\delta_Y}'C E^\delta_X \\ ({E^\delta_Y}'C E^\delta_X)' & -{E^\delta_X}'(A_s+\gamma_T' \gamma_T)E^\delta_X \end{array}\right]
=\\
&
\left[\begin{array}{@{}cc@{}} \identity & 0 \\({E^\delta_Y}'C E^\delta_X)'  ({E^\delta_Y}'A_s E^\delta_Y)^{-1} & \identity\end{array}\right] \circ\\
&
\left[\begin{array}{@{}cc@{}} {E^\delta_Y}'A_s E^\delta_Y & 0 \\ 0  &
-{E^\delta_X}'(A_s+\gamma_T' \gamma_T)E^\delta_X-({E^\delta_Y}'C E^\delta_X)' ({E^\delta_Y}'A_s E^\delta_Y)^{-1}{E^\delta_Y}'C E^\delta_X
\end{array}\right] \circ\\
&
\left[\begin{array}{@{}cc@{}} \identity & ({E^\delta_Y}'A_s E^\delta_Y)^{-1}{E^\delta_Y}'C E^\delta_X \\0   & \identity\end{array}\right].
\end{split}
\ee
We conclude that this Galerkin operator is invertible if and only if 
the Schur complement
\be \label{schur}
{E^\delta_X}'(A_s+\gamma_T' \gamma_T)E^\delta_X+({E^\delta_Y}'C E^\delta_X)' ({E^\delta_Y}'A_s E^\delta_Y)^{-1}{E^\delta_Y}'C E^\delta_X
\ee
is invertible, which holds true for any $X^\delta \neq \{0\}$.

\begin{theorem} \label{thm:main} Let $(Y^\delta,X^\delta)_{\delta \in \Delta}$ be a family of closed subspaces of $Y \times X$ such that 
\be \label{infsup}
\gamma_\Delta:=\inf_{\delta\in \Delta}\inf_{\{u \in X^\delta\colon \partial_t u \neq 0\}} \sup_{0\neq v \in Y^\delta} \frac{(\partial_t u)(v)}{\|\partial_t u\|_{Y'}\|v\|_Y} >0.
\footnote{Here and in the following, $\inf_{\{u \in X^\delta\colon \partial_t u \neq 0\}} \sup_{0\neq v \in Y^\delta} \frac{(\partial_t u)(v)}{\|\partial_t u\|_{Y'}\|v\|_Y}$ should be read as $1$
in the case that $\{u \in X^\delta\colon \partial_t u \neq 0\}=\emptyset$.}
\ee 
Let $\rho=\rho_\Delta$ be the root in $[0,1)$ of
$$
\gamma_\Delta^2 (\rho^2-\rho)+\|A_a\|_{\cL(Y,Y')}^2(\rho-1)+\rho=0,
$$
and let
$$
C_\Delta:=\frac{(3+\|A_a\|_{\cL(Y,Y')}^2)(\sqrt{3}+\|A_a\|_{\cL(Y,Y')})}{(1-\rho_\Delta)\gamma_\Delta^2},
$$
so that $C_\Delta=3 \sqrt{3} \,\gamma_\Delta^{-2}$ when $\|A_a\|_{\cL(Y,Y')}=0$, and $\lim_{\|A_a\|_{\cL(Y,Y')} \rightarrow \infty} C_\Delta=\infty$.
Then with $\lambda=u$ and $(\lambda^\delta,u^\delta)$ denoting the solutions of \eqref{var3} and its Galerkin discretization, respectively, it holds that
\be \label{21}
\sqrt{\|\lambda -\lambda^\delta\|^2_Y+\|u -u^\delta\|^2_X} \leq C_\Delta 
  \inf_{(\bar{\lambda}^\delta,\bar{u}^\delta) \in Y^\delta\times X^\delta} \sqrt{\|\lambda -\bar{\lambda}^\delta\|^2_Y+\|u -\bar{u}^\delta\|^2_X}.
\ee
\end{theorem}


\begin{proof} 
In view of the second inequality presented in Remark~\ref{remmie}, we start with bounding the norm of the continuous operator.
Using Young's inequality, for $(\lambda,u) \in Y \times X$ we have
\begin{align*}
&\|A_s\lambda +\partial_t u\|_{Y'}^2+\|\partial_t' \lambda -(A_s+\gamma_T' \gamma_T)u\|^2_{X'} \\
& \leq
{\textstyle \frac{3}{2}} \|A_s\lambda\|_{Y'}^2+3\|\partial_t u\|_{Y'}^2+
{\textstyle \frac{3}{2}}\|\partial_t' \lambda\|^2_{X'}+3\|(A_s+\gamma_T' \gamma_T)u\|^2_{X'}\\
&\leq
{\textstyle \frac{3}{2}}(\|\lambda\|_{Y}^2+\|\lambda\|_{Y}^2)+3 (\|\partial_t u\|_{Y'}^2+\|u\|^2_Y+\|u(T)\|^2)=3(\|\lambda\|_Y^2+\|u\|_{X}^2).
\end{align*}
Together with $\|A_a u\|_{Y'}^2+\|A_a' \lambda\|_{X'}^2\leq \|A_a\|^2_{\cL(Y,Y')}(\|\lambda\|_Y^2+\|u\|_X^2)$, it shows that
\begin{align*}
&\Big\| \left[\begin{array}{@{}cc@{}} A_s & C \\ C'  & -(A_s+\gamma_T' \gamma_T) \end{array}\right]\Big\|_{\cL(Y \times X,Y' \times X')}\\
&\leq
\Big\| \left[\begin{array}{@{}cc@{}} A_s & \partial_t \\ \partial_t '  & -(A_s+\gamma_T' \gamma_T) \end{array}\right]\Big\|_{\cL(Y \times X,Y' \times X')}+
\Big\| \left[\begin{array}{@{}cc@{}} 0 & A_a \\ A_a'  & 0 \end{array}\right]\Big\|_{\cL(Y \times X,Y' \times X')}\\
&\leq \sqrt{3}+\|A_a\|_{\cL(Y,Y')}.
\end{align*}

To bound, in view of \eqref{stablePG3}, the norm of the inverse of the Galerkin operator, we use the block-LDU factorization \eqref{factorized}.
With $r:=(1+\|A_a\|_{\cL(Y,Y')}^2)$, for $u \in X$ it holds that
$$
\|C u\|_{Y'} \leq \|\partial_t u\|_{Y'}+\|A_a\|_{\cL(Y,Y')} \|u\|_{Y} \leq \sqrt{r}\, \|u\|_X.
$$
Together with the fact that ${E_Y^\delta}' A_s E_Y^\delta \in \Lis(Y^\delta,{Y^\delta}')$ is an isometry and again Young's inequality, it shows that for $(\lambda,u) \in Y^\delta\times X^\delta$,
\begin{align*}
\|\lambda-({E_Y^\delta}' A_s E_Y^\delta)^{-1} {E_Y^\delta}' C E_X^\delta u\|_Y^2+\|u\|_X^2 &\leq (1+r)\|\lambda\|_Y^2+(1+r^{-1}) r \|u\|_X^2+ \|u\|_X^2\\
&\leq (2+r)(\|\lambda\|_Y^2+\|u\|_X^2),
\end{align*}
or
$$
\Big\|
\left[\begin{array}{@{}cc@{}} \identity & ({E^\delta_Y}'A_s E^\delta_Y)^{-1}{E^\delta_Y}'C E^\delta_X \\0   & \identity\end{array}\right]^{-1}
\Big\|_{\cL(Y^\delta \times X^\delta,Y^\delta \times X^\delta)}
\leq \sqrt{3+\|A_a\|_{\cL(Y,Y')}^2}.
$$
Obviously, the $\cL({Y^\delta}' \times {X^\delta}',{Y^\delta}' \times {X^\delta}')$-norm of the inverse of the first factor at the right-hand side of \eqref{factorized} satisfies the same bound.

Moving to the second factor, 
we consider the Schur complement operator.
From $({E_Y^\delta}' A_s E_Y^\delta \lambda)(\lambda) =\|\lambda\|^2_Y$ for $\lambda \in Y^\delta$, we have for $f \in {Y^\delta}'$,  $f(({E_Y^\delta}' A_s E_Y^\delta)^{-1}f) =\|({E_Y^\delta}' A_s E_Y^\delta)^{-1} f\|^2_Y=\|f\|_{{Y^\delta}'}^2$, and so for $u \in X^\delta$
$$
\Big(({E^\delta_Y}'C E^\delta_X)' ({E^\delta_Y}'A_s E^\delta_Y)^{-1}{E^\delta_Y}'C E^\delta_X u\Big)(u)=\|{E^\delta_Y}'C E^\delta_X u\|_{{Y^\delta}'}^2.
$$
Using that for $u \in X^\delta$,
$$
\|{E_Y^\delta}' \partial_t E_X^\delta u\|_{{Y^\delta}'}^2=\Big(\sup_{0 \neq v \in Y^\delta}\frac{(\partial_t u)(v)}{\|v\|_Y}\Big)^2 \geq \gamma_\Delta^2 \|\partial_t u\|_{Y'}^2
$$
and
$$
\|{E_Y^\delta}' A_a E_X^\delta u\|_{{Y^\delta}'}^2 \leq \|A_a\|_{\cL(Y,Y')}^2 \|u\|_{Y}^2,
$$
Young's inequality shows that                                
\begin{align*}
\|{E^\delta_Y}'C E^\delta_X u\|_{{Y^\delta}'}^2
\geq (1-\rho_\Delta) \gamma_\Delta^2 \|\partial_t u\|_{Y'}^2+(1-\rho_\Delta^{-1}) \|A_a\|_{\cL(Y,Y')}^2 \|u\|_{Y}^2,
\end{align*}
where we assumed that $\rho_\Delta>0$ i.e.~$A_a \neq 0$.
It follows that
\begin{align}\nonumber
((A_s+&\gamma_T' \gamma_T)u)(u)+\|{E^\delta_Y}'C E^\delta_X u\|_{{Y^\delta}'}^2
\\\nonumber
&\geq(1+(1-\rho_\Delta^{-1}) \|A_a\|_{\cL(Y,Y')}^2)\|u\|_Y^2+\|u(T)\|^2+(1-\rho_\Delta) \gamma_\Delta^2 \|\partial_t u\|_{Y'}^2\\ \label{extra}
&\geq (1-\rho_\Delta)\gamma_\Delta^2 \|u\|_X^2
\end{align}
where we used that $1+(1-\rho_\Delta^{-1}) \|A_a\|_{\cL(Y,Y')}^2=(1-\rho_\Delta) \gamma_\Delta^2$ by definition of $\rho_\Delta$.
One easily verifies \eqref{extra} also in the case that $A_a=0$ i.e.~$\rho_\Delta=0$.

Since ${E_Y^\delta}' A_s E_Y^\delta \in \Lis(Y^\delta,{Y^\delta}')$ is an isometry, and $0<(1-\rho_\Delta)\gamma_\Delta^2 \leq \gamma_\Delta^2 \leq 1$, we conclude that the $\cL({Y^\delta}'\times{X^\delta}', Y^\delta\times X^\delta)$-norm of the inverse of the second factor is bounded by $(1-\rho_\Delta)^{-1}\gamma_\Delta^{-2}$.

In view of the second inequality presented in Remark~\ref{remmie} in combination with \eqref{stablePG3}, 
the proof is completed by collecting the bounds that were derived.
\end{proof}

\subsection{Galerkin discretizations of \eqref{var4}} 
Although it is likely possible to generalize results to the case of $A_a \neq 0$, as in \cite{11,249.2} in this section we operate under the condition that
\be \label{symm}
A=A_s.
\ee
Following \cite{249.2}, for a given closed subspace $Y^\delta \subseteq Y$ we define the `mesh-dependent' norm on $X$ by
$$
\|u\|_{X,Y^\delta}^2:=\|u\|_{Y}^2 +\sup_{0 \neq v \in Y^\delta} \frac{(\partial_t u)(v)^2}{\|v\|_{Y}^2}  + \|u(T)\|^2.
$$
Note that $\|\,\|_{X,Y}=\|\,\|_{X}$.

The following result generalizes the `inf-sup identity', known for $Y^\delta=Y$, see e.g.~\cite{70.95}, to mesh-dependent norms.

\begin{lemma} \label{infsupidentity}
Assuming \eqref{symm}, then for $u \in Y^\delta \cap X$, 
$$
\|u\|_{X,Y^\delta}^2=\sup_{0 \neq v \in Y^\delta} \frac{(Bu)(v)^2}{\|v\|_{Y}^2}+\|u(0)\|^2.
$$
If additionally $\gamma_0 u\in H^\delta$, then 
\be \label{infsup3}
\|u\|_{X,Y^\delta}^2=\sup_{0 \neq (v_1,v_2) \in Y^\delta \times H^\delta} \frac{((Bu)(v_1)+\langle u(0),v_2\rangle)^2}{\|v_1\|_{Y}^2+\|v_2\|^2}.
\ee
\end{lemma}

\begin{proof}
Let $y \in Y^\delta$ be defined by $(A_s y)(v)=(\partial_t u)(v)$ ($v \in Y^\delta$).
Then $(A_s y)(y)=\sup_{0 \neq v \in Y^\delta} \frac{(\partial_t u)(v)^2}{\|v\|_{Y}^2}$.
Furthermore, for $v \in Y^\delta$, $(Bu)(v)=(A_s(y+u))(v)$ and so, thanks to $u \in Y^\delta$,
\begin{align*}
\sup_{0 \neq v \in Y^\delta} \frac{(Bu)(v)^2}{\|v\|_{Y}^2}&=(A_s(y+u))(y+u)=
(A_s y)(y)+2(A_s y)(u)+ (A_s u)(u)\\
&=
(A_s y)(y)+2(\partial_t u)(u)+ (A_s u)(u)=
\|u\|_{X,Y^\delta}^2-\|u(0)\|^2
\end{align*}
where we used that $2\int_I  \langle \partial_t u(t),u(t)\rangle\,dt=\|u(T)\|^2-\|u(0)\|^2$.

The second statement follows from 
$$
\sup_{0 \neq (v_1,v_2) \in Y^\delta \times H^\delta} \frac{((A_s(y+u))(v_1)+\langle u(0),v_2\rangle)^2}{\|v_1\|_{Y}^2+\|v_2\|^2}=
(A_s(y+u))(y+u)+\|u(0)\|^2,
$$
thanks to $u(0) \in H^\delta$.
\end{proof}

The next theorem gives sufficient conditions for existence and uniqueness of solutions of the Galerkin discretization of \eqref{var4}, and provides a suboptimal error estimate.

\begin{theorem} \label{suboptimal} Assuming \eqref{symm}, for closed subspaces $Y^\delta \times H^\delta \times X^\delta \subset Y \times H \times X$ with $X^\delta \subseteq Y^\delta$ and $\ran \gamma_0|_{X^\delta}\subseteq H^\delta$, the Galerkin discretization of \eqref{var4} has a unique solution $(\mu^\delta,\sigma^\delta,u^\delta) \in Y^\delta \times H^\delta \times X^\delta $, and with $u$ denoting the solution of \eqref{exact},
$$
\|u-u^\delta\|_{X,Y^\delta} \leq 2 \inf_{\bar{u}^\delta \in X^\delta} \|u-\bar{u}^\delta\|_{X}.
$$
\end{theorem}

\begin{proof}
Thanks to the assumptions $X^\delta \subseteq Y^\delta$ and $\ran \gamma_0|_{X^\delta}\subseteq H^\delta$,
the inf-sup identity  \eqref{infsup3} quarantees the unique solvability of the Galerkin system.

For any $u \in X^\delta$, there exist unique $y_u \in Y^\delta$, $h_u \in H^\delta$ such that
$$
(A_s y_u)(v_1)+\langle h_u,v_2\rangle=(Bu)(v_1)+\langle \gamma_0 u,v_2\rangle\quad((v_1,v_2) \in Y^\delta\times H^\delta).
$$
We decompose $Y^\delta \times H^\delta$ into $Z^\delta:=\clos\{(y_u,h_u)\colon u  \in X^\delta\}$\footnote{In the (discontinuous) Petrov-Galerkin community, $Y^\delta \times H^\delta$ and $Z^\delta$ are known under the names test search space (or search test space), and projected optimal test space (or approximate optimal test space), respectively.}
and its orthogonal complement $W^\delta$. Using that for any $u \in X^\delta$ and $(v_1,v_2) \in W^\delta$, $(B u)(v_1)+\langle u(0),v_2\rangle=0$, one infers that
for any $u \in X^\delta$, the inf-sup identity \eqref{infsup3} remains valid when the supremum is restricted to $0 \neq (v_1,v_2) \in Z^\delta$.
Furthermore, since for any $(v_1,v_2) \in Z^\delta$ there exists a $z \in X^\delta$ with $(B z)(v_1)+\langle z(0),v_2\rangle \neq 0$, 
we infer that
$u^\delta$ is the unique solution of the Petrov-Galerkin discretization of finding $u^\delta \in X^\delta$ such that
\be \label{PG}
(B u^\delta)(v_1)+\langle u^\delta(0),v_2\rangle=g(v_1)+\langle u_0,v_2\rangle \quad((v_1,v_2) \in Z^\delta). 
\ee
By applying both these observations consecutively, we infer that for any $\bar{u}^\delta \in X^\delta$,
\be \label{estP}
\begin{split}
&\|u^\delta-\bar{u}^\delta\|^2_{X,Y^\delta}
= \sup_{0 \neq (v_1,v_2) \in Z^\delta} \frac{((B(u^\delta-\bar{u}^\delta))(v_1)+\langle u^\delta(0)-\bar{u}^\delta(0),v_2\rangle)^2}{\|v_1\|_{Y}^2+\|v_2\|^2}\\
&= \sup_{0 \neq (v_1,v_2) \in Z^\delta} \frac{((B(u-\bar{u}^\delta))(v_1)+\langle u(0)-\bar{u}^\delta(0),v_2\rangle)^2}{\|v_1\|_{Y}^2+\|v_2\|^2}
\leq 
\|u-\bar{u}^\delta\|^2_{X},
\end{split}
\ee
where we again applied \eqref{infsup3} now for $Y^\delta=Y$. A triangle-inequality completes the proof.
\end{proof}

Theorem~\ref{suboptimal} can be used to demonstrate optimal rates for the error in $u^\delta$ in the $\|\,\|_{X,Y^\delta}$-norm, and hence also in the $Y$-norm.
Yet, for doing so one needs to control the error of best approximation in the generally strictly stronger $\|\,\|_X$-norm, which requires 
 regularity conditions on the solution $u$ that exceeds those that are needed to guarantee optimal rates 
 of the best approximation in the $\|\,\|_{X,Y^\delta}$-norm.
In other words, this theorem does not show that $u^\delta$ is a quasi-best approximation to $u$ from $X^\delta$ in the $\|\,\|_{X,Y^\delta}$-norm, or in any other norm.

\begin{remark} \label{steinbach1}
Theorem~\ref{suboptimal} provides a generalization, with an improved constant, of Steinbach's result \cite[Theorem 3.2]{249.2}.
There the case was considered that the initial value $u_0=0$, $\ran \gamma_0|_{X^\delta}=\{0\}$,  $H^\delta=\{0\}$, and $Y^\delta=X^\delta$.
 In that case the Galerkin discretization of \eqref{var4} means solving $u^\delta \in X^\delta$ from 
$(B u^\delta)(v)=g(v)$ ($v \in X^\delta$) (indeed, $Z^\delta$ in the proof of Theorem~\ref{suboptimal} is $X^\delta \times \{0\}$).
So with this approach the forming of `normal equations' as in \eqref{var5} is avoided.

In case of an inhomogeneous initial value $u_0 \in H$, one may approximate the solution as $\bar{u}+w^\delta$, where $\bar{u} \in X$ is such that $\gamma_0 \bar{u}=u_0$, and $w^\delta \in X^\delta$ solves
$(B w^\delta)(v)=g(v)-(B \bar{u})(v)$ ($v \in X^\delta$). Although such a $\bar{u} \in X$ always exists, its practical construction becomes inconvenient for $u_0 \not\in V$.
For $u_0 \in V$, $\bar{u}$ can be taken as its constant extension in time.

To investigate in the setting of \cite{249.2} the relation between the $\|\,\|_{X,X^\delta}$- and $\|\,\|_X$-norms, we consider $X^\delta$ of the form $X^\delta_t \otimes X^\delta_x$, where $X^\delta_t$ is the space of continuous piecewise linears, zero at $t=0$, w.r.t.~a uniform partition of $I$ with mesh-size $h_\delta=\frac{T}{2 N_\delta}$ for some $N_\delta \in \N$, and $X^\delta_x \subset V$ with $\cap_{\delta \in \Delta} X_x^\delta \neq \{0\}$.
Given $z^\delta \in X^\delta$, Lemma~\ref{infsupidentity} shows that
\be \label{20}
\sup_{0 \neq v \in X^\delta} \frac{|(B z^\delta)(v)|}{\|z^\delta\|_X\|v\|_Y}= \frac{\|z^\delta\|_{X,X^\delta}}{\|z^\delta\|_X}.
\ee

For some arbitrary, fixed $0 \neq z_x \in \cap_{\delta \in \Delta} X_x^\delta$, we take $z^\delta=z^\delta_t \otimes z_x \in X^\delta$, where $z^\delta_t \in X^\delta_t$ is defined by $\frac{d}{d t} z^\delta_t=(-1)^{i-1}$ on $[(i-1) h_\delta,i h_\delta]$. Since $z^\delta_t(0)=0$, also $z^\delta_t(T)=0$.
We have $\|z^\delta_t\|_{L_2(I)}\eqsim h_\delta$, $\|\frac{d z^\delta_t}{d t}\|_{L_2(I)}\eqsim 1$,
$\sup_{0 \neq v \in Y}\frac{(\partial_t z^\delta)(v)}{\|v\|_Y}=\|\frac{d z^\delta_t}{d t}\|_{L_2(I)}\|z_x\|_{V'} \eqsim 1$,
$\|z^\delta\|_Y =\|z^\delta_t\|_{L_2(I)} \|z_x\|_{V} \eqsim h_\delta$, and
\begin{align*}
\sup_{0 \neq v \in X^\delta}\frac{(\partial_t z^\delta)(v)}{\|v\|_Y}&=
\sup_{0 \neq v \in X_t^\delta}\frac{\langle\frac{dz^\delta_t }{dt}, v \rangle_{L_2(I)}}{\|v\|_{L_2(I)}}
\sup_{0 \neq v \in X_x^\delta}\frac{\langle z_x,  v \rangle}{\|v\|_{V}}\\
& \leq \sup_{0 \neq v \in X_t^\delta}\frac{\langle\frac{dz^\delta_t }{dt}, v \rangle_{L_2(I)}}{\|v\|_{L_2(I)}} \|z_x\|_{V'}.
\end{align*}

Let us equip the space of piecewise constants w.r.t.~the aforementioned uniform partition
with the $L_2(I)$-normalized basis $\{\chi_i^\delta\}$ of characteristic functions of the subintervals, and $X^\delta_t$ with the set of nodal basis functions $\{\phi_i^\delta\}$ normalized such that their maximal value is $h_\delta^{-\frac12}$. Then
with
$G:=[\langle \chi_j,\phi_i\rangle_{L_2(I)}]_{i j}
=\frac12 {\small \left[\begin{array}{@{}rlcl@{}l@{}}
1 & 1 & & \\
&\ddots& \ddots &\\
&  & 1& 1\\
&  & & 1
\end{array}
\right]}$, and $\vec{x}:=\sqrt{h_\delta}\, [(-1)^{i-1}]_{1 \leq i \leq 2 N_\delta}$, from the uniform $L_2(I)$-stability of $\{\phi_i^\delta\}$ one infers that
$$
\sup_{0 \neq v \in X_t^\delta}\frac{\langle\frac{dz^\delta_t }{dt}, v \rangle_{L_2(I)}}{\|v\|_{L_2(I)}} \eqsim
\sup_{0 \neq \vec{y}}\frac{\langle G \vec{x},\vec{y}\rangle}{\|\vec{y}\|} = \|G \vec{x}\| =\frac12 \sqrt{h_\delta}.
$$
By substituting these estimates in the right-hand side of \eqref{20}, we find that its value is $\eqsim \sqrt{h_\delta}$, so that
$\inf_{0 \neq z^\delta \in X^\delta} \sup_{0 \neq v \in X^\delta} \frac{|(B z^\delta)(v)|}{\|z^\delta\|_X\|v\|_Y} \lesssim \sqrt{h_\delta}$.
As follows from the first inequality in Remark~\ref{remmie}, this means that there exist solutions $u \in X$ of the parabolic problem for which the errors in $X$-norm in these Galerkin approximations from $X^\delta$ are a factor $\gtrsim h_\delta^{-\frac12}$ larger than these errors in the best approximations from $X^\delta$.

Numerical evidence provided by \cite[Table 6]{249.2} indicate that in general these Galerkin approximations are not quasi-optimal in the $Y$-norm either.
\end{remark}

Returning to the general setting of Theorem~\ref{suboptimal},
in the following theorem it will be shown that under an \emph{additional} assumption 
quasi-optimal error estimates are valid.

\begin{theorem} \label{andreev} Assuming \eqref{symm}, let $(Y^\delta,H^\delta,X^\delta)_{\delta \in \Delta}$ be a family of closed subspaces of $Y \times H \times X$ such that in addition to
$X^\delta \subseteq Y^\delta$ and $\ran \gamma_0|_{X^\delta}\subseteq H^\delta$, also \eqref{infsup} is valid.
Then for the Galerkin solutions $(\mu^\delta,\sigma^\delta,u^\delta) \in Y^\delta \times H^\delta \times X^\delta $ of \eqref{var4} it holds that
$$
\|u-u^\delta\|_X \leq \gamma_\Delta^{-1} \inf_{\bar{u}^\delta \in X^\delta}\|u-\bar{u}^\delta\|_X.
$$
\end{theorem}

\begin{proof} As we have seen in the proof of Theorem~\ref{suboptimal}, thanks to the assumptions $X^\delta \subseteq Y^\delta$ and $\ran \gamma_0|_{X^\delta}\subseteq H^\delta$, the component $u^\delta \in X^\delta$ of the Galerkin solution of \eqref{var4} is the Petrov-Galerkin solution of \eqref{exact} with test space $Z^\delta \subset Y^{\delta} \times H^{\delta}$.

Equation \eqref{estP} shows that the projector $P^\delta\colon u \mapsto u^\delta$ satisfies $\|P^\delta u\|_{X,Y^\delta} \leq \|u\|_X$.
The proof is completed by $\|\,\|_X \leq \gamma_\Delta^{-1} \|\,\|_{X,Y^\delta}$ on $X^\delta$ by assumption \eqref{infsup}, in combination with \eqref{stablePG1}.
\end{proof}

In \cite{11}, Andreev studied minimal residual Petrov-Galerkin discretizations of $\left[\begin{array}{@{}c@{}} B \\ \gamma_0\end{array}\right]u=\left[\begin{array}{@{}c@{}} g \\ \gamma_0' u_0\end{array}\right]$.
They can equivalently be interpreted as 
Galerkin discretizations of \eqref{var4} (cf. \cite{45.44}, \cite[Prop. 2.2]{35.8565}).
In view of this, Theorem~\ref{andreev} reproduces, though here with a clear-cut constant, the results from
 \cite[Thms.~3.1 \& 4.1]{11}.

\begin{remark} As was pointed out earlier in \cite{11}, for practical computations it can be attractive to modify the Galerkin discretization of \eqref{var4} by replacing ${E_Y^\delta}' A_s E_Y^\delta$ by some  $\tilde{A}_s^\delta={\raisebox{0pt}[0pt][0pt]{$\tilde{A}$}_s^\delta}' \in \Lis(Y^\delta,{Y^\delta}')$ whose inverse can be determined cheaply (a preconditioner)
\footnote{For Galerkin discretizations of \eqref{var3}, such a replacement of ${E_Y^\delta}' A_s E_Y^\delta$ by an equivalent operator will result in an inconsistent discretization.}, such that for some constants $0<c_{\mathcal N} \leq C_{\mathcal N}<\infty$,
$$
\frac{(\tilde{A}_s^\delta u)(u)}{(A_su)(u)} \in [c_{\mathcal N}^2,C_{\mathcal N}^2] \quad (\delta \in \Delta,\,u \in Y^\delta).
$$
Indeed, in that case one can solve the then explicitly available Schur complement equation with precondition CG, instead of applying the preconditioned MINRES iteration.
By redefining $Z^\delta:=\clos_{Y^\delta \times H^\delta} \ran\left[\begin{array}{@{}c@{}} (\tilde{A}_s^\delta)^{-1} {E_Y^\delta}' B \\ \gamma_0 \end{array}\right]\Big|_{X^\delta}$  in the proof of Theorem~\ref{suboptimal},
and by taking $W^\delta$ to be its orthogonal complement in $Y^\delta \times H^\delta$ with $Y^\delta$ now being equipped with inner product $(\tilde{A}_s^\delta \cdot)(\cdot)$, 
instead of \eqref{estP} we now estimate for any $\bar{u}^\delta \in X^\delta$ ,
\begin{align*}
\|u^\delta-\bar{u}^\delta\|^2_{X,Y^\delta}
& = \sup_{0 \neq (v_1,v_2) \in Y^\delta} \frac{((B(u^\delta-\bar{u}^\delta))(v_1)+\langle u^\delta(0)-\bar{u}^\delta(0),v_2\rangle)^2}{\|v_1\|_{Y}^2+\|v_2\|^2}\\
&\leq {\textstyle\frac{1}{\min(c_{\mathcal N}^2,1)}}
\sup_{0 \neq (v_1,v_2) \in Y^\delta}\frac{((B(u^\delta-\bar{u}^\delta))(v_1)+\langle u^\delta(0)-\bar{u}^\delta(0),v_2\rangle)^2}{(\tilde{A}_s^\delta v_1)(v_1)^2+\|v_2\|^2}\\
&=
{\textstyle\frac{1}{\min(c_{\mathcal N}^2,1)}}
\sup_{0 \neq (v_1,v_2) \in Z^\delta} \frac{((B(u^\delta-\bar{u}^\delta))(v_1)+\langle u^\delta(0)-\bar{u}^\delta(0),v_2\rangle)^2}{(\tilde{A}_s^\delta v_1)(v_1)^2+\|v_2\|^2}\\
&=
{\textstyle\frac{1}{\min(c_{\mathcal N}^2,1)}}
\sup_{0 \neq (v_1,v_2) \in Z^\delta} \frac{((B(u-\bar{u}^\delta))(v_1)+\langle u(0)-\bar{u}^\delta(0),v_2\rangle)^2}{(\tilde{A}_s^\delta v_1)(v_1)^2+\|v_2\|^2}\\
&\leq
{\textstyle\frac{\max(C_{\mathcal N}^2,1)}{\min(c_{\mathcal N}^2,1)}}
\sup_{0 \neq (v_1,v_2) \in Z^\delta} \frac{((B(u-\bar{u}^\delta))(v_1)+\langle u(0)-\bar{u}^\delta(0),v_2\rangle)^2}{\|v_1\|_{Y}^2+\|v_2\|^2}\\
&\leq
{\textstyle \frac{\max(C_{\mathcal N}^2,1)}{\min(c_{\mathcal N}^2,1)}} \|u-\bar{u}^\delta\|^2_{X}.
\end{align*}
Consequently, a generalization of the statement of Theorem~\ref{suboptimal} reads as
$$
\|u-u^\delta\|_{X,Y^\delta} \leq \Big(1+\sqrt{{\textstyle \frac{\max(C_{\mathcal N}^2,1)}{\min(c_{\mathcal N}^2,1)}}}\,\Big) \inf_{\bar{u}^\delta \in X^\delta} \|u-\bar{u}^\delta\|_{X},
$$
and that of Theorem~\ref{andreev} as
$$
\|u-u^\delta\|_{X} \leq \gamma_\Delta^{-1} \sqrt{{\textstyle \frac{\max(C_{\mathcal N}^2,1)}{\min(c_{\mathcal N}^2,1)}}} \inf_{\bar{u}^\delta \in X^\delta} \|u-\bar{u}^\delta\|_{X}.
$$
\end{remark}

\begin{remark} As we have seen in the previous section, under the condition that \eqref{infsup} is valid, Galerkin discretizations of \eqref{var3} yield quasi-optimal approximations.
Assuming $A=A'$, in the current section we have seen that the same holds true for Galerkin discretizations of \eqref{var4} when in addition $X^\delta \subseteq Y^\delta$ and $\ran \gamma_0|_{X^\delta}\subseteq H^\delta$.
For the latter discretization, however, a still suboptimal error bound is valid without assuming \eqref{infsup}.
This raises the question whether this is also true for Galerkin discretizations of \eqref{var3}.

As we have seen earlier, the Galerkin operator resulting from of \eqref{var3} is invertible whenever $X^\delta\neq \{0\}$.
Moreover, when equipping $X^\delta$ with the `mesh-dependent' norm $\|\,\|_{X,Y^\delta}$, by adapting the proof of Theorem~\ref{thm:main} 
one can show that the Galerkin operator is in $\Lis(Y^\delta \times X^\delta, {Y^\delta}' \times {X^\delta}')$ with both the operator and its inverse having a uniformly bounded norm.
Despite this result, we could not establish, however, a suboptimal error estimate similar to Theorem~\ref{suboptimal}.
\end{remark}

Finally in this section we comment on the implementation of the Galerkin discretization of \eqref{var4}.
This system reads as 
\be \label{e0}
\left[\begin{array}{@{}ccc@{}} {E_Y^\delta}' A_s E_Y^\delta& 0 & {E_Y^\delta}' B E_X^\delta
\\ 0 & {E_H^\delta}'  E_H^\delta& {E_H^\delta}'\gamma_0 E_X^\delta\\
{E_X^\delta}' B' E_Y^\delta& {E_X^\delta}' \gamma_0' E_H^\delta& 0\end{array}\right]
\left[\begin{array}{@{}c@{}} \mu^\delta \\ \sigma^\delta \\ u^\delta \end{array}\right]=
\left[\begin{array}{@{}c@{}} {E_Y^\delta}' g \\ {E_H^\delta}' u_0 \\ 0 \end{array}\right],
\ee
By eliminating $\sigma^\delta$, it is equivalent to
\be \label{e1}
\left[\begin{array}{@{}cc@{}} {E_Y^\delta}' A_s E_Y^\delta & {E_Y^\delta}' B E_X^\delta\\
{E_X^\delta}' B' E_Y^\delta& -{E_X^\delta}' \gamma_0' E_H^\delta \big({E_H^\delta}' E_H^\delta\big)^{-1} {E_H^\delta}'\gamma_0 E_X^\delta\end{array}\right]
\left[\begin{array}{@{}c@{}} \mu^\delta  \\ u^\delta \end{array}\right]=
\left[\begin{array}{@{}c@{}} {E_Y^\delta}' g \\ -{E_X^\delta}' \gamma_0' u_0  \end{array}\right].
\ee
The operator $E_H^\delta \big({E_H^\delta}' E_H^\delta\big)^{-1} {E_H^\delta}'$ is the $H$-orthogonal projector onto $H^\delta$.
So under the assumption that $$\ran \gamma_0|_{X^\delta}\subseteq H^\delta$$ which was made in Theorem~\ref{andreev}, it can be omitted, or equivalently, it can be pretended that $H^\delta=H$, without changing the solution $(\mu^\delta,u^\delta)$.
The implementation of the resulting system
\be \label{e2}
\left[\begin{array}{@{}cc@{}} {E_Y^\delta}' A_s E_Y^\delta & {E_Y^\delta}' B E_X^\delta\\
{E_X^\delta}' B' E_Y^\delta& -{E_X^\delta}' \gamma_0' \gamma_0 E_X^\delta\end{array}\right]
\left[\begin{array}{@{}c@{}} \mu^\delta  \\ u^\delta \end{array}\right]=
\left[\begin{array}{@{}c@{}} {E_Y^\delta}' g \\ -{E_X^\delta}' \gamma_0' u_0  \end{array}\right].
\ee
is easier, and it runs more efficiently than \eqref{e0}.

\begin{remark}
The system \eqref{e2}  can be viewed as a Galerkin discretisation of
\be \label{e3}
\left[\begin{array}{@{}cc@{}} A_s  &  B \\
 B' & - \gamma_0' \gamma_0 \end{array}\right]
\left[\begin{array}{@{}c@{}} \mu  \\ u \end{array}\right]=
\left[\begin{array}{@{}c@{}} g \\ - \gamma_0' u_0  \end{array}\right],
\ee
but for the analysis of the discretization error in $(\mu^\delta,u^\delta)$ it is still useful to view \eqref{e2} before elimination of $\sigma^\delta$, as a Galerkin discretization of \eqref{var4} 
which yielded the sharp bound on this error presented in Theorem~\ref{andreev}.
\end{remark}

\section{Realization of the uniform inf-sup stability \eqref{infsup}} \label{S4}
In Theorem~\ref{thm:main} it was shown that Galerkin discretizations of \eqref{var3} are quasi-optimal when \eqref{infsup} is valid, and in Theorem~\ref{andreev} the same was shown for Galerkin discretizations of \eqref{var4} when in addition 
$X^\delta \subseteq Y^\delta$ and $\ran \gamma_0|_{X^\delta} \subseteq H^\delta$ (and $A=A_s$) are valid.

In this section we realize the condition \eqref{infsup} for finite element spaces w.r.t.~partitions of the space-time domain into prismatic elements.
In \S\ref{Slocalglobal} generally non-uniform partitions are considered for which the partition in time is independent of the spatial location, and the spatial mesh in each time slab is such that the corresponding $H$-orthogonal projection is uniformly $V$-stable.
In \S\ref{Sglobalglobal} we revisit the special case, already studied in \cite{11}, of trial spaces that are tensor products of temporal and spatial trial spaces.

\subsection{Non-uniform approximation in space \emph{local} in time, non-uniform approximation in time \emph{global} in space} \label{Slocalglobal}
\begin{theorem} \label{infsupverification}
Let ${\mathcal O}$ be a collection of closed subspaces $X_x$ of $V$ such that the $H$-orthogonal projector $Q_{X_x}$ onto $X_x$ is in $\cL(V,V)$, with $\mu_{\mathcal O} := \inf_{X_x \in {\mathcal O}} \|Q_{X_x}\|_{\cL(V,V)}^{-1}>0$.
For any $N \in \N$, $0=t_0<t_1<\cdots<t_N=T$,  $q_0,\ldots,q_{N-1} \in \N$, $X_x^0,\ldots,X_x^{N-1} \in {\mathcal O}$, let
\begin{align*}
X^\delta &:=\{u \in C(\bar{I};V) \colon u|_{(t_i,t_{i+1})} \in P_{q_i} \otimes X_x^i\}\\
Y^\delta &:=\{v \in L_2(I;V) \colon v|_{(t_i,t_{i+1})} \in P_{q_i-1} \otimes X_x^i\}
\end{align*}
Then with $\Delta$ being the collection of all $\delta=\delta(N,(t_i)_i, (q_i)_i, (X_x^i)_i)$,  it holds that
\be \label{infsup2}
\inf_{\delta \in \Delta}\inf_{\{u \in X^\delta \colon \partial_t u \neq 0\}} \sup_{0\neq v \in Y^\delta} \frac{(\partial_t u)(v)}{\|\partial_t u\|_{Y'}\|v\|_Y}\geq \mu_{\mathcal O},
\ee
i.e.~\eqref{infsup} is valid.
\end{theorem}

\begin{proof} In \cite[Lemma 6.2]{11} it was shown that $\inf_{0 \neq u \in X_x} \sup_{0 \neq v \in X_x} \frac{\langle u,v\rangle}{\|u\|_{V'}\|v\|_{V}}=\|Q_{X_x}\|_{\cL(V,V)}^{-1}$. 

With $P_n$ denoting the Legendre polynomial of degree $n$, extended with zero outside $(-1,1)$, for any $u \in X^\delta $, $\partial_t u$ can be written as the $L_2(I;H)$-orthogonal expansion $(t,x) \mapsto \sum_{i=0}^{N-1} \sum_{n=0}^{q_i-1} P_n\big(\frac{2t-(t_{i+1}+t_i)}{t_{i+1}-t_i}\big) u_{i,n}(x)$ for some $u_{i,n} \in X_x^i$. 
Fixing $\eps \in (0,\mu_{\mathcal O} )$, for each $(i,n)$ there is a $v_{i,n} \in X_x^i$ with $\|v_{i,n}\|_{V}=\|u_{i,n}\|_{V'}$ and $\langle u_{i,n},v_{i,n}\rangle \geq (\mu_{\mathcal O} -\eps) \|u_{i,n}\|_{V'} \|v_{i,n}\|_{V}$.
Taking $v:=(t,x) \mapsto \sum_{i=0}^{N-1} \sum_{n=0}^{q_i-1}  P_n\big(\frac{2t-(t_{i+1}+t_i)}{t_{i+1}-t_i}\big) v_{i,n}(x)$, we conclude that
$$
(\partial_t u)(v) \geq (\mu_{\mathcal O} -\eps) \sum_{i=0}^{N-1} \sum_{n=0}^{q_i-1} \big\|P_n\big({\textstyle \frac{2\cdot-(t_{i+1}+t_i)}{t_{i+1}-t_i}}\big)\big\|_{L_2(I)}^2 \|u_{i,n}\|_{V'}^2=(\mu_{\mathcal O} -\eps) \|u\|_{Y'} \|v\|_Y,
$$
which implies the result.
\end{proof}

\begin{remark}
In view of Theorem~\ref{andreev}, note that both $X^\delta \subset Y^\delta$ and \eqref{infsup} are valid by taking 
$Y^\delta :=\{v \in L_2(I;V) \colon v|_{(t_i,t_{i+1})} \in P_{q_i} \otimes X_x^i\}$.
\end{remark}

Considering the condition on the collection ${\mathcal O}$ of spatial trial spaces $X_x$, let us consider the typical situation that $H=L_2(\Omega)$, $V=H^1_{0,\gamma}(\Omega)=\{u \in H^1(\Omega)\colon u=0 \text{ on } \gamma\}$ where $\Omega \subset \R^d$ is a bounded polytopal domain, and $\gamma$ is a measurable, closed, possibly empty subset of $\partial\Omega$.
We consider $X_x \subset V$ to be finite element spaces of some degree w.r.t.~a family of uniformly shape regular, and, say, conforming partitions ${\mathcal T}$ of $\Omega$ into, say, $d$-simplices, where $\gamma$ is the union of some $(d-1)$-faces of $S \in {\mathcal T}$.
When the partitions in this family are
quasi-uniform, then using e.g.~the Scott-Zhang quasi-interpolator (\cite{247.2}), it is easy to demonstrate the
so-called (uniform) \emph{simultaneous approximation property}
$$
\sup_{X_x \in {\mathcal O}} \sup_{0 \neq u \in V} \frac{\inf_{v \in X_x}\{\|v\|_V+\big(\sup_{0\neq w \in X_x} \frac{\|w\|_V}{\|w\|_H}\big) \|u-v\|_H\}}{\|u\|_V}<\infty.
$$
Writing for $u \in V$ and any $v \in X_x$, $Qu=v+Q(u-v)$, one easily infers that
$\sup_{X_x \in {\mathcal O}} \|Q_x\|_{\cL(V,V)}<\infty$.

The uniform boundedness of $\|Q_x\|_{\cL(V,V)}$ is, however, by no means restricted to families of finite element spaces w.r.t.~quasi-uniform partitions, and it has been demonstrated for families of locally refined partitions, for $d=2$ including those that are generated by the newest vertex bisection algorithm. We refer to \cite{37.2,75.3675}. 

\subsection{Non-uniform approximation in space \emph{global} in time, non-uniform approximation in time \emph{global} in space} \label{Sglobalglobal}
If in Theorem~\ref{infsupverification}, the spatial trial spaces $X_x^i$ are independent of the temporal interval $(t_i,t_{i+1})$, then $X^\delta$ is a tensor product of trial spaces in space and time. In that case, one shows inf-sup stability for general temporal trial spaces, e.g.~spline spaces with more global smoothness than continuity.

\begin{theorem} \label{globalglobal} Let ${\mathcal O}$ be as in Theorem~\ref{infsupverification}. Given closed subspaces $X_t \subset H^1(I)$, $\frac{d}{dt} X_t \subseteq Y_t \subset L_2(I)$ and $X_x \in {\mathcal O}$, let $X^\delta:=X_t \otimes X_x$, $Y^\delta:=Y_t \otimes X_x$. Then with $\Delta$ being the collection of all $\delta=\delta(X_t,Y_t,X_x)$, \eqref{infsup2} is valid.
\end{theorem}

The proof of this result follows from the fact that thanks to the Kronecker product structure of  $\partial_t \in \cL(X,Y')$, for such trial spaces we have
\begin{align} \nonumber
&\inf_{\{u \in X^\delta\colon \partial_t u \neq 0\}} \sup_{0\neq v \in Y^\delta} \frac{(\partial_t u)(v)}{\|\partial_t u\|_{Y'}\|v\|_Y}
\\ \label{factorization}
&=\inf_{\{u \in X_t\colon \frac{d u}{d t} \neq 0\}} \sup_{0\neq v \in Y_t} \frac{\int_I \frac{d u}{d t} v\,dt}{\|\frac{d u}{d t}\|_{L_2(I)}\|v\|_{L_2(I)}}
\times \inf_{0\neq u \in X_x} \sup_{0\neq v \in X_x} \frac{\langle u,v\rangle}{\|u\|_{V'}\|v\|_{V}}\\ \nonumber
& =\inf_{0\neq u \in X_x} \sup_{0\neq v \in X_x} \frac{\langle u,v\rangle}{\|u\|_{V'}\|v\|_{V}}.
\end{align}
(To see this, one may use that for Hilbert spaces $U$ and $V$, $T \in \cL(U,V')$, and Riesz mappings $R_U\colon U \rightarrow U'$, $R_V\colon V \rightarrow V'$, it holds that
$\inf_{0 \neq u \in U}\sup_{0 \neq v \in V}\frac{(Tu)(v)}{\|u\|_U\|v\|_V}=\min \sigma(R_U^{-1} T' R_V^{-1} T)$, with $R_U^{-1} T' R_V^{-1} T\in \cL(U,U)$ being self-adjoint and non-negative.
In the above setting, it is a Kronecker product of corresponding operators acting in the `time' and `space' direction, respectively.)

\begin{remark}[Sparse tensor products]
Instead of considering the `full' tensor product trial spaces from Theorem~\ref{globalglobal}, more efficient approximations can be found by the application of  `sparse' tensor products.
 Let $X_x^{(0)}\subset X_x^{(1)} \subset \cdots $ be a sequence of spaces from ${\mathcal O}$, $X_t^{(0)}\subset X_t^{(1)}\subset \cdots \subset H^1(I)$, and $Y_t^{(0)}\subset Y_t^{(1)}\subset \cdots \subset L_2(I)$ such that $Y_t^{(k)} \supseteq \frac{d}{d t}X_t^{(k)}$. Then for 
$X^{(\ell)}:=\sum_{k=0}^\ell X_t^{(k)} \otimes X_x^{(\ell-k)}$,
$Y^{(\ell)}:=\sum_{k=0}^\ell Y_t^{(k)} \otimes X_x^{(\ell-k)}$
inf-sup stability holds true uniformly in $\ell$ with inf-sup constant $\mu_{\mathcal O}$.

Although this result follows as a special case from the analysis given in \cite{11} for convenience we include the argument.
Defining $W_t^{(k)}:=Y_t^{(k)} \cap (Y_t^{(k-1)})^{\perp_{L_2(I)}}$ for $k>0$, and $W_t^{(0)}:=Y_t^{(0)}$,
from the nestings of $(Y_t^{(i)})_i$ and  $(X_x^{(i)})_i$ one infers that
$Y^{(\ell)}=\oplus_{k=0}^\ell W_t^{(k)} \otimes X_x^{(\ell-k)}$ is
an $(L_2(I)\otimes H)$-orthogonal decomposition.
Given $y \in Y^{(\ell)}$, let $y=\sum_{k=0}^\ell y_k$ be the corresponding expansion. Fixing $\eps \in (0,\mu_{\mathcal O} )$,
there exist $\tilde{y}_k \in W_t^{(k)} \otimes X_x^{(\ell-k)}$ with
$\langle y_k,\tilde{y}_k\rangle_{L_2(I)\otimes H} \geq (\mu_{\mathcal O}-\eps)\|y_k\|_Y\|\tilde{y}_k\|_{Y'}$ and $\|\tilde{y}_k\|_{Y'}=\|y_k\|_Y$, and so
$\langle \sum_{k=0}^\ell y_k,\sum_{k=0}^\ell \tilde{y}_k\rangle_{L_2(I)\otimes H} \geq
(\mu_{\mathcal O}-\eps) \|\sum_{k=0}^\ell y_k\|_Y\|\sum_{k=0}^\ell \tilde{y}_k\|_{Y'}$. 
Thanks to $\partial_t X^{(\ell)} \subseteq Y^{(\ell)}$, the proof  is completed.
\end{remark}

\begin{remark} In view of \eqref{factorization}, it is obvious that Theorem~\ref{globalglobal} remains valid when the condition $\frac{d}{dt} X_t \subseteq Y_t$ is relaxed to
$\inf_{\{u \in X_t\colon \frac{d u}{d t} \neq 0\}} \sup_{0\neq v \in Y_t} \frac{\int_I \frac{d u}{d t} v\,dt}{\|\frac{d u}{d t}\|_{L_2(I)}\|v\|_{L_2(I)}}>0$ uniformly in the pairs $(X_t,Y_t)$ that are applied.
As shown in \cite{11}, the same holds true in the sparse tensor product case.
For $X_t$ being the space of continuous piecewise linears w.r.t.~some partition $\tria$ of $I$, and $Y_t$ being the space of continuous piecewise linears w.r.t.~the once dyadically refined partition, an easy computation shows that the inf-sup constant is not less than $\sqrt{3/4}$. 

Since in our experiments with the method from \cite{11}, with this alternative choice of $Y_t$ the numerical results are slightly better than when taking $Y_t$ to be the space of discontinuous piecewise linears w.r.t.~$\tria$, we will report on results obtained with this alternative choice for $Y_t$.
\end{remark}

\section{Numerical experiments} \label{Snum}
For the simplest possible case of the heat equation in one space dimension discretized using as `primal' trial space $X^\delta$ the space of continuous piecewise bilinears w.r.t.~a uniform partition into squares, we compare the accuracy of approximations provided by the newly proposed method (i.e.~the Galerkin discretization of \eqref{var3} with trial space here denoted by $Y_{\text{new}}^\delta \times X^\delta$) with those obtained with the method from \cite{11} (i.e.~the Galerkin discretization of \eqref{var4}).
We implement the latter method in the form \eqref{e2}, i.e. after eliminating $\sigma^\delta$.
The remaining trial space is denoted here by $Y_{\text{Andr.}}^\delta \times X^\delta$.
So we take $T=1$, i.e.~$I=(0,1)$, and with $\Omega:=(0,1)$, $H:=L_2(\Omega)$, $V:=H^1_0(\Omega)$, $a(t;\eta,\zeta):=\int_\Omega \eta' \zeta'\,dx$. 
With $\frac{1}{h_t}=\frac{1}{h_x} \in \N$, we set
\begin{alignat*}{3}
X^\delta&:=&&\{v \in H^1(I)\colon v|_{(ih_t,(i+1)h_t)} \in P_1\} &&\otimes \{v \in H_0^1(\Omega)\colon v|_{(ih_x,(i+1)h_x)} \in P_1\},\\
Y_{\text{new}}^\delta&:=&&\{v \in L_2(I)\colon v|_{(ih_t,(i+1)h_t)} \in P_0\} &&\otimes \{v \in H_0^1(\Omega)\colon v|_{(ih_x,(i+1)h_x)} \in P_1\},\\
Y_{\text{Andr}}^\delta&:=&&\{v \in H^1(I)\colon v|_{(ih_t/2,(i+1)h_t/2)} \in P_1\} && \otimes\{v \in H_0^1(\Omega)\colon v|_{(ih_x,(i+1)h_x)} \in P_1\},
\end{alignat*}
Note that $\dim Y^\delta_{\text{new}} \approx \dim X^\delta$ and $\dim Y^\delta_{\text{Andr}} \approx 2\dim X^\delta$.
The total number of non-zeros in the whole system matrix of the new method is asymptotically a factor 2 smaller than this number for Andreev's method.

Prescribing both a smooth exact solution $u(t,x)=e^{-2t} \sin \pi x$ and a singular one $u(t,x)=e^{-2t} |t-x| \sin \pi x$, Figure~\ref{fig1} shows the errors $e^\delta := u - u^\delta$ in $X$-norm as a function of $\dim X^\delta$.
\begin{figure}[h]
  \begin{center}
    \includegraphics[width=0.45\linewidth]{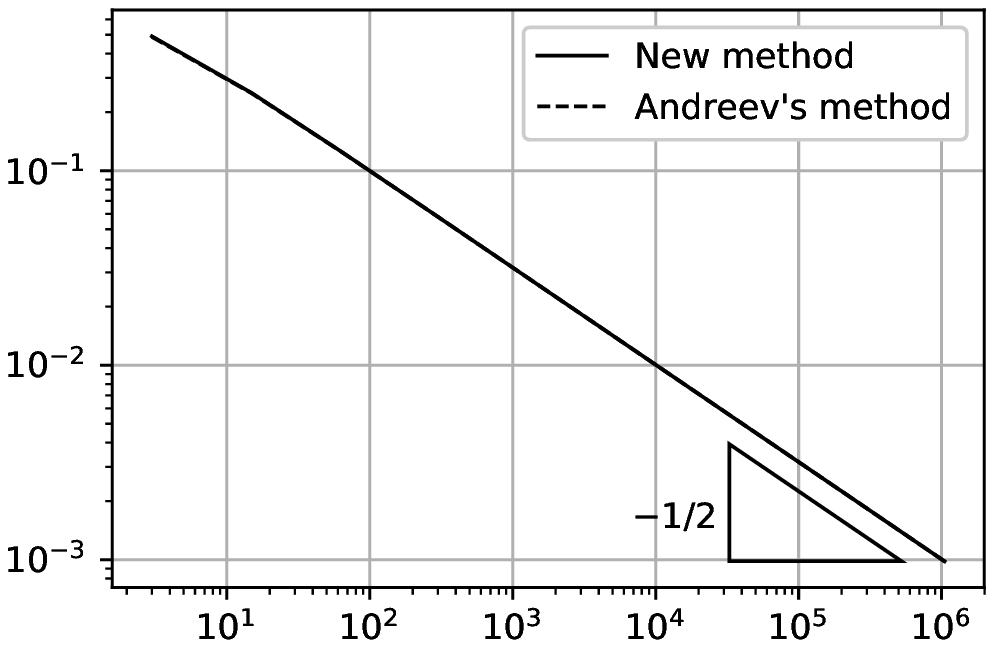}
    \includegraphics[width=0.45\linewidth]{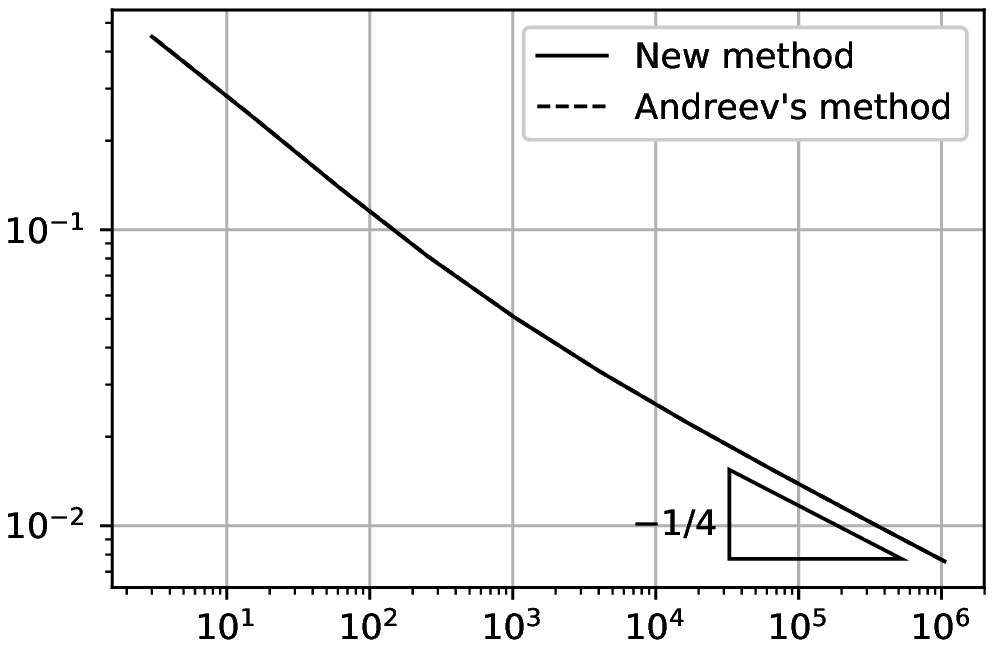}
  \end{center}
  \caption{$\|e^\delta\|_X$ vs.~$\dim X^\delta$ for both numerical methods. Left: $u(t,x)=e^{-2t} \sin \pi x$. Right: $u(t,x)=e^{-2t} |t-x| \sin \pi x$.}
  \label{fig1}
\end{figure}
The norms of the errors in the Galerkin solutions found by the two methods are nearly indistinguishable from one another.
Furthermore, the observed convergence rates $1/2$ and $1/4$, respectively, are the best possible ones that in view of the polynomial degrees of $X^\delta$ and $Y^\delta$ (new method) or that of $X^\delta$ (Andreev's method) and the regularity of the solutions can be expected with the application of uniform meshes. (For any $\eps>0$, $e^{-2t} |t-x| \sin \pi x \in H^{\frac32-\eps}(I \times \Omega) \setminus H^{\frac32}(I \times \Omega)$).

For both solutions and both numerical methods, the errors $e^\delta(T,\cdot)$ measured in $L_2(\Omega)$ converge with the better rate $1$, i.e., these errors are asymptotically proportional to $h_x^2=h_t^2$, see left picture in Figure~\ref{fig2}. To illustrate that the two methods yield different Galerkin solutions, we show $e^\delta(0, \cdot)$, measured in $L_2(\Omega)$-norm in the right of Figure~\ref{fig2}.
\begin{figure}[h]
  \begin{center}
    \includegraphics[width=0.45\linewidth]{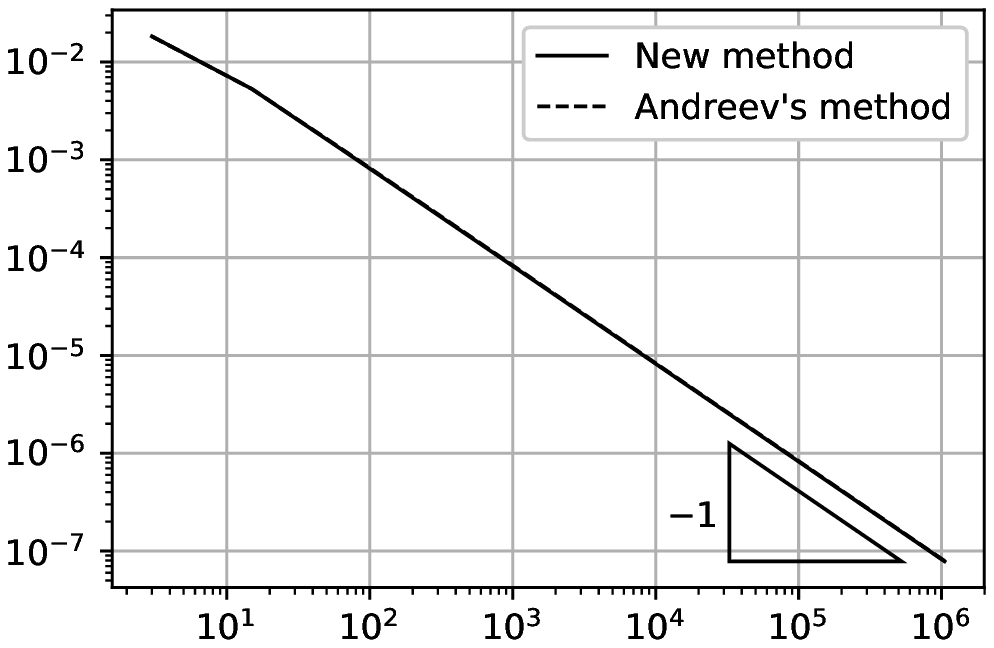}
    \includegraphics[width=0.45\linewidth]{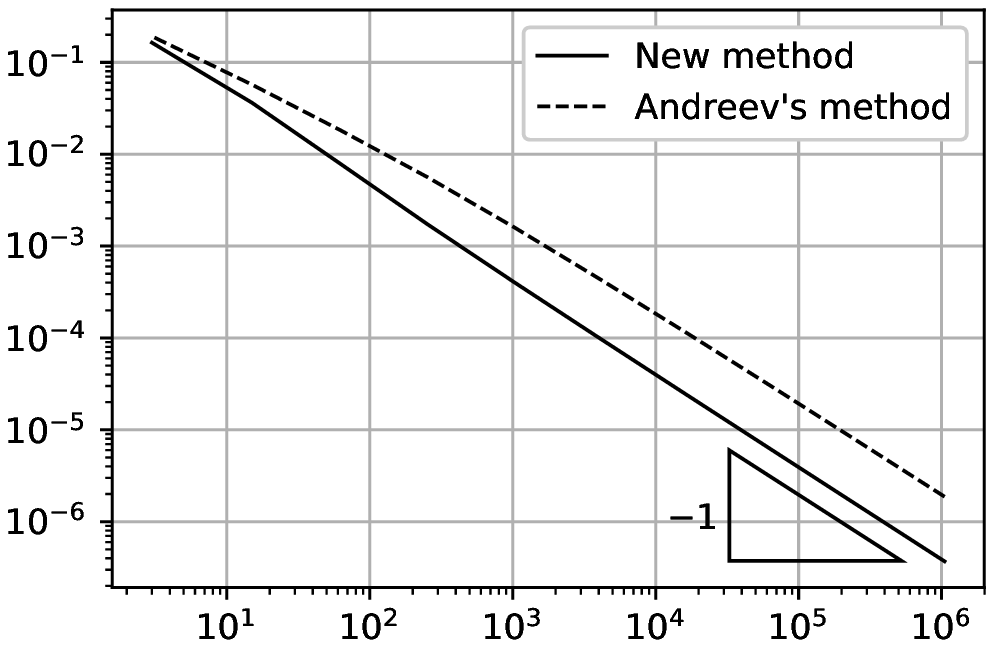}
  \end{center}
  \caption{Singular solution $u(t,x)=e^{-2t} |t-x| \sin \pi x$. Left: $\|e^\delta(T,\cdot)\|_{L_2(\Omega)}$ vs.~$\dim X^\delta$. Right: $\|e^\delta(0,\cdot)\|_{L_2(\Omega)}$ vs.~$\dim X^\delta$.}
  \label{fig2}
\end{figure}

The new method actually yields two approximations for $u$, viz.~$u^\delta$ and $\lambda^\delta$. This secondary approximation is not in $X$, but it is in $Y=L_2(I;V)$.
For both solutions, the errors in $\lambda^\delta$ measured in $Y$-norm are slightly larger than in those in $u^\delta$, see left picture in Figure~\ref{fig3}.

Finally, we replaced the symmetric spatial diffusion operator by a nonsymmetric convection-diffusion operator
$a(t;\eta,\zeta):=\int_\Omega \eta' \zeta'\,+ \beta \eta' \zeta dx$.
Letting $\beta := 100$ and again taking the singular solution $u(t,x)=e^{-2t} |t-x| \sin \pi x$, the errors $e^\delta$ in $X$-norm of both Galerkin solutions vs.~$\dim X^\delta$ are given in Figure~\ref{fig3}. We once again see that the two methods show very comparable convergence behaviour.
\begin{figure}[h]
  \begin{center}
    \includegraphics[width=0.45\linewidth]{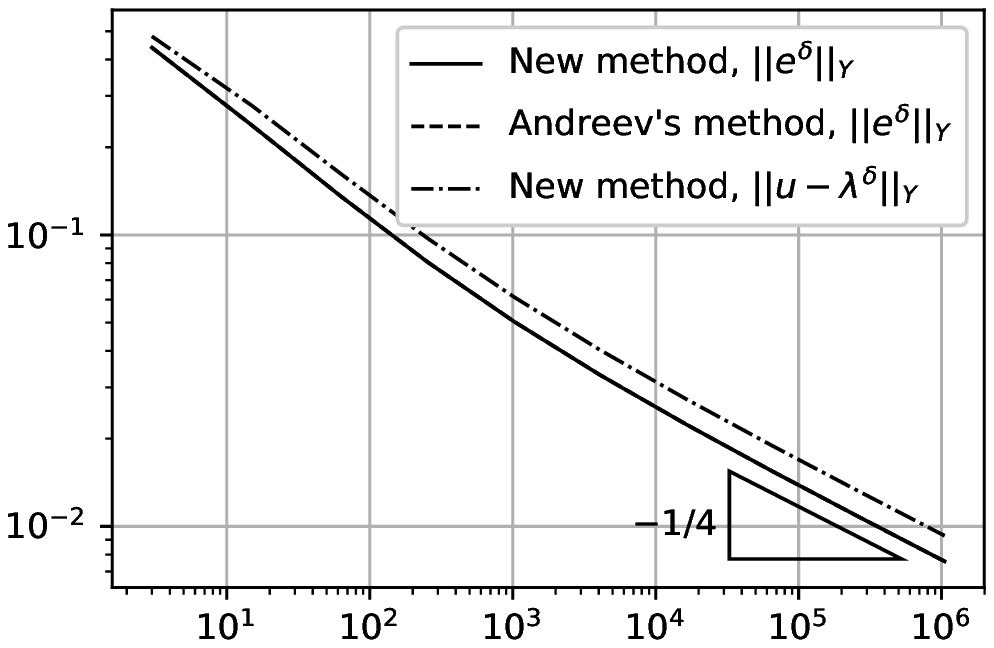}
    \includegraphics[width=0.45\linewidth]{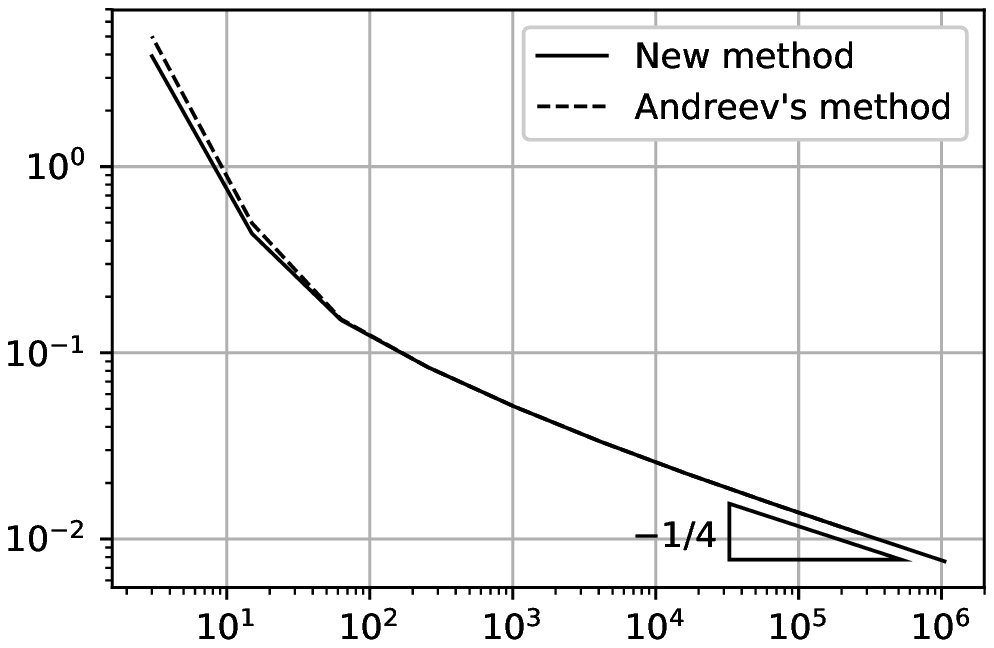}
  \end{center}
  \caption{Singular solution $u(t,x) = e^{-2t} |t-x| \sin \pi x$. Left: $\|e^\delta\|_Y$ and $\|u-\lambda^\delta\|_Y$ vs.~$\dim X^\delta$ for the symmetric problem. Right: $\|e^\delta\|_X$ vs.~$\dim X^\delta$ for the nonsymmetric problem.}
  \label{fig3}
\end{figure}

\section{Conclusion} \label{Sconcl}
Three related (Petrov-) Galerkin discretizations of space-time variational formulations were analyzed.
The Galerkin scheme introduced by Steinbach in \cite{249.2} has the lowest computational cost, and applies on general space-time meshes, but depending on the exact solution, the numerical solutions can be far from quasi-optimal in the natural mesh-independent norm.
The minimal residual Petrov-Galerkin discretization introduced by Andreev in \cite{11} yields for suitable trial and test pairs quasi-optimal approximations from the trial space.
For suitable pairs of trial spaces, Galerkin discretizations of a newly introduced mixed space-time variational formulation also yield quasi-optimal approximations, 
but for the same accuracy at a lower computational cost than with the method from \cite{11}.


\begin{thebibliography}{CDW12}

\bibitem[And12]{10.6}
R.~Andreev.
\newblock {\em {Stability of space-time Petrov-Galerkin discretizations for
  parabolic evolution equations}}.
\newblock PhD thesis, {ETH Z\"{u}rich}, 2012.

\bibitem[And13]{11}
R.~Andreev.
\newblock Stability of sparse space-time finite element discretizations of
  linear parabolic evolution equations.
\newblock {\em IMA J. Numer. Anal.}, 33(1):242--260, 2013.

\bibitem[And16]{12.5}
R.~Andreev.
\newblock Wavelet-in-time multigrid-in-space preconditioning of parabolic
  evolution equations.
\newblock {\em SIAM J. Sci. Comput.}, 38(1):A216--A242, 2016.

\bibitem[BE76]{35.81}
H.~Br\'{e}zis and I.~Ekeland.
\newblock Un principe variationnel associ\'{e} \`a certaines \'{e}quations
  paraboliques. {L}e cas d\'{e}pendant du temps.
\newblock {\em C. R. Acad. Sci. Paris S\'{e}r. A-B}, 282(20):Ai, A1197--A1198,
  1976.

\bibitem[BJ89]{18.63}
I.~Babu{\v{s}}ka and T.~Janik.
\newblock The {$h$}-{$p$} version of the finite element method for parabolic
  equations. {I}. {T}he {$p$}-version in time.
\newblock {\em Numer. Methods Partial Differential Equations}, 5(4):363--399,
  1989.

\bibitem[BJ90]{18.64}
I.~Babu{\v{s}}ka and T.~Janik.
\newblock The {$h$}-{$p$} version of the finite element method for parabolic
  equations. {II}. {T}he {$h$}-{$p$} version in time.
\newblock {\em Numer. Methods Partial Differential Equations}, 6(4):343--369,
  1990.

\bibitem[BS14]{35.8565}
D.~Broersen and R.P. Stevenson.
\newblock A robust {P}etrov-{G}alerkin discretisation of convection-diffusion
  equations.
\newblock {\em Comput. Math. Appl.}, 68(11):1605--1618, 2014.

\bibitem[Car02]{37.2}
C.~Carstensen.
\newblock Merging the {B}ramble-{P}asciak-{S}teinbach and the
  {C}rouzeix-{T}hom\'ee criterion for {$H^1$}-stability of the
  {$L^2$}-projection onto finite element spaces.
\newblock {\em Math. Comp.}, 71(237):157--163, 2002.

\bibitem[CDW12]{45.44}
A.~Cohen, W.~Dahmen, and G.~Welper.
\newblock Adaptivity and variational stabilization for convection-diffusion
  equations.
\newblock {\em {ESAIM: Mathematical Modelling and Numerical Analysis}},
  46:1247--1273, 2012.

\bibitem[DL92]{63}
R.~Dautray and J.-L. Lions.
\newblock {\em Mathematical analysis and numerical methods for science and
  technology. {V}ol. 5}.
\newblock Springer-Verlag, Berlin, 1992.
\newblock Evolution problems I.

\bibitem[DS18]{64.577}
D.~Devaud and Ch. Schwab.
\newblock Space-time {$hp$}-approximation of parabolic equations.
\newblock {\em Calcolo}, 55(3):Art. 35, 23, 2018.

\bibitem[Dup82]{70.1}
T.~Dupont.
\newblock Mesh modification for evolution equations.
\newblock {\em Math. Comp.}, 39(159):85--107, 1982.

\bibitem[ESV17]{70.95}
A.~Ern, I.~Smears, and M.~Vohral\'{i}k.
\newblock Guaranteed, locally space-time efficient, and polynomial-degree
  robust a posteriori error estimates for high-order discretizations of
  parabolic problems.
\newblock {\em SIAM J. Numer. Anal.}, 55(6):2811--2834, 2017.

\bibitem[FK19]{75.257}
T.~F\"{u}hrer and M.~Karkulik.
\newblock Space-time least-squares finite elements for parabolic equations.
\newblock Technical report, 2019.
\newblock arXiv:1911.01942.

\bibitem[GHS16]{75.3675}
F.~D. Gaspoz, C.-J. Heine, and K.~G. Siebert.
\newblock Optimal grading of the newest vertex bisection and {$H^1$}-stability
  of the {$L_2$}-projection.
\newblock {\em IMA J. Numer. Anal.}, 36(3):1217--1241, 2016.

\bibitem[GK11]{77.5}
M.D. Gunzburger and A.~Kunoth.
\newblock Space-time adaptive wavelet methods for control problems constrained
  by parabolic evolution equations.
\newblock {\em {SIAM J. Contr. Optim.}}, 49(3):1150--1170, 2011.

\bibitem[GN16]{75.27}
M.J. Gander and M.~Neum\"uller.
\newblock Analysis of a new space-time parallel multigrid algorithm for
  parabolic problems.
\newblock {\em SIAM J. Sci. Comput.}, 38(4):A2173--A2208, 2016.

\bibitem[Kat60]{169.5}
T.~Kato.
\newblock Estimation of iterated matrices, with application to the von
  {N}eumann condition.
\newblock {\em Numer. Math.}, 2:22--29, 1960.

\bibitem[LMN16]{169.05}
U.~Langer, S.E. Moore, and M.~Neum\"uller.
\newblock Space-time isogeometric analysis of parabolic evolution problems.
\newblock {\em Comput. Methods Appl. Mech. Engrg.}, 306:342--363, 2016.

\bibitem[Nay76]{233.5}
B.~Nayroles.
\newblock Deux th\'{e}or\`emes de minimum pour certains syst\`emes dissipatifs.
\newblock {\em C. R. Acad. Sci. Paris S\'{e}r. A-B}, 282(17):Aiv, A1035--A1038,
  1976.

\bibitem[NS19]{234.7}
M.~Neum\"{u}ller and I.~Smears.
\newblock Time-parallel iterative solvers for parabolic evolution equations.
\newblock {\em {SIAM J. Sci. Comput.}}, 41(1):C28--C51, 2019.

\bibitem[RS18]{243.867}
N.~{Rekatsinas} and R.~{Stevenson}.
\newblock An optimal adaptive tensor product wavelet solver of a space-time
  fosls formulation of parabolic evolution problems.
\newblock {\em Adv. Comput. Math.}, 2018.

\bibitem[SS09]{247.15}
Ch. Schwab and R.P. Stevenson.
\newblock A space-time adaptive wavelet method for parabolic evolution
  problems.
\newblock {\em Math. Comp.}, 78:1293--1318, 2009.

\bibitem[SS17]{247.155}
Ch. Schwab and R.P. Stevenson.
\newblock Fractional space-time variational formulations of {(Navier)-Stokes}
  equations.
\newblock {\em SIAM J. Math. Anal.}, 49(4):2442--2467, 2017.

\bibitem[Ste15]{249.2}
O.~Steinbach.
\newblock Space-{T}ime {F}inite {E}lement {M}ethods for {P}arabolic {P}roblems.
\newblock {\em Comput. Methods Appl. Math.}, 15(4):551--566, 2015.

\bibitem[SZ90]{247.2}
L.~R. Scott and S.~Zhang.
\newblock Finite element interpolation of nonsmooth functions satisfying
  boundary conditions.
\newblock {\em Math. Comp.}, 54(190):483--493, 1990.

\bibitem[SZ18]{249.3}
O.~Steinbach and M.~Zank.
\newblock Coercive space-time finite element methods for initial boundary value
  problems.
\newblock Berichte aus dem {I}nstitut f{\uumlaut}r {A}ngewandte {M}athematik,
  {B}ericht 2018/7, Technische Universit{\aumlaut}t Graz, 2018.

\bibitem[TV16]{258.4}
F.~Tantardini and A.~Veeser.
\newblock The {$L^2$}-projection and quasi-optimality of {G}alerkin methods for
  parabolic equations.
\newblock {\em SIAM J. Numer. Anal.}, 54(1):317--340, 2016.

\bibitem[UP14]{299}
K.~Urban and A.~T. Patera.
\newblock An improved error bound for reduced basis approximation of linear
  parabolic problems.
\newblock {\em Math. Comp.}, 83(288):1599--1615, 2014.

\bibitem[VR18]{310.6}
I.~Voulis and A.~Reusken.
\newblock A time dependent {S}tokes interface problem: well-posedness and
  space-time finite element discretization.
\newblock {\em ESAIM Math. Model. Numer. Anal.}, 52(6):2187--2213, 2018.

\bibitem[Wlo82]{314.9}
J.~Wloka.
\newblock {\em Partielle {D}ifferentialgleichungen}.
\newblock B. G. Teubner, Stuttgart, 1982.
\newblock Sobolevr\"aume und Randwertaufgaben.

\bibitem[XZ03]{315.7}
J.~Xu and L.~Zikatanov.
\newblock Some observations on {B}abu\v ska and {B}rezzi theories.
\newblock {\em Numer. Math.}, 94(1):195--202, 2003.

\end{thebibliography}

\end{document}